\documentclass[graybox]{svmult}

\usepackage{type1cm}        
\usepackage{makeidx}         
\usepackage{graphicx}        
\usepackage{multicol}        
\usepackage[bottom]{footmisc}
\usepackage{tikz,pgfplots}
\usepackage{stackrel}

\usepackage{newtxtext}       %
\usepackage[varvw]{newtxmath}       

\makeindex             

\begin{document}

\title*{Conformally symplectic Dynamics}
\author{Marie-Claude ARNAUD}
\institute{Marie-Claude ARNAUD \at Universit\'e Paris Cit\'e, Sorbonne Universit\'e, CNRS, IMJ-PRG, F-75013 Paris, France Supported by ANR CoSyDy, ANR-21-CE40-0014-01. \email{arnaud@imj-prg.fr}}%
%
\maketitle

\abstract*{Each chapter should be preceded by an abstract (no more than 200 words) that summarizes the content. The abstract will appear \textit{online} at \url{www.SpringerLink.com} and be available with unrestricted access. This allows unregistered users to read the abstract as a teaser for the complete chapter.
Please use the 'starred' version of the \texttt{abstract} command for typesetting the text of the online abstracts (cf. source file of this chapter template \texttt{abstract}) and include them with the source files of your manuscript. Use the plain \texttt{abstract} command if the abstract is also to appear in the printed version of the book.}

\abstract{Dynamists have been studying Hamiltonian systems for a long time. However, many physical systems are dissipative and do not preserve a symplectic form. \\
This is the case, for example, with systems involving friction, which multiply the symplectic form by a constant smaller than 1.  We will prove that almost every point is in the unstable set of infinity for these systems and we will illustrate different situations that may arise with examples.
We will also study invariant manifolds by such dynamics. We will provide an example where an invariant proper submanifold is not isotropic and give different conditions that imply that a given invariant submanifold is isotropic. In particular, we will outline a strange link between isotropy and entropy. Examples demonstrate that some systems have a global attractor, while others do not. We will give a sufficient condition for a conformally Hamiltonian dynamics of a cotangent bundle to have a global attractor. Then we will introduce two very classical examples~: Ma$\tilde{\text{n}}$\'e example and damped mechanical systems.\\
After that,  we introduce the notion of locally symplectic manifold. Unlike symplectic manifolds, these manifolds carry conformally symplectic dynamics that have a conservative part and a dissipative part. For  conformally Hamiltonian flow, we will describe dissipative and conservative orbits using their number of rotation.
}

\section{Conformally symplectic Dynamics on a symplectic manifold~: an introduction}
\label{sec:1}
In this section, we present the framework, give several examples, and prove that for a
conformally symplectic dynamics, almost all Lebesgue points lie in the unstable set
of infinity. We also explain why the study of conformally symplectic flows reduces to the
study of Liouville flows.

 We assume that $(M^{(2d)}, \omega)$ is a {\sl symplectic manifold}, i.e. $M$ is an even-dimensional manifold endowed with a closed non-degenerate 2-form. 
\begin{remark}\phantom{fish}
\begin{itemize}
\item[$\bullet$] Darboux Theorem implies the existence of (symplectic) coordinates $(q, p)=(q_1, \dots, q_d, p_1, \dots, p_d)$ such that $$\omega=dq\wedge dp=\sum_{i=1}^ddq_i\wedge dp_i.$$
\item[$\bullet$] Sometimes, we will assume that $M$ is {\sl exact symplectic}, which means that there is a $1$-form $\lambda$ such that $\omega=-d\lambda$. Then $\lambda$ is called a {\sl Liouville $1$-form}.
\end{itemize}
\end{remark}
\begin{example}\phantom{fish}
\begin{itemize}
\item[$\bullet$] The usual Liouville form on $\mathbb R^{2d}$ is $\lambda=\sum_{i=1}^dp_idq_i$  where \\$(q_1, \dots, q_d, p_1, \dots, p_d)\in\mathbb R^{2d}$. Hence the corresponding symplectic form is $\omega=\sum_{i=1}^d dq_i\wedge dp_i$.
\item[$\bullet$] On an orientable (closed) surface, an area form is a non-exact symplectic form.
\item[$\bullet$] Let $N^{(d)}$ be a closed manifold. The cotangent bundle $T^*N$ is endowed with the tautological Liouville 1-form $\lambda$. If $q_1, \dots, q_d$ are coordinates on $N$ and $p_1, \dots p_d$ are the coordinates on $T^*_qM$ associated to the basis $(dq_1, \dots, dq_d)$, we have $\lambda=\sum_{i=1}^d p_idq_i$. Note that this is independent of the chosen coordinates $(q_1, \dots, q_d)$ .
\end{itemize}
\end{example}
\begin{definition} A conformally symplectic dynamics is
\begin{itemize}
\item[$\bullet$] either a diffeomorphism $f: M\to M$ such that there exists $a\in (0, 1)\cup (1, +\infty)$ so that $f^*\omega=a\, \omega$; then $a$ is called the {\sl conformality ratio}.
\item[$\bullet$] or a vector field $X$ on $M$ such that there exists $\alpha \in \mathbb R^*$ satisfying $L_X\omega=\alpha\, \omega$, where $L_X$ is the Lie derivative. Then the flow $(\varphi_t)$ satisfies $\varphi_t^*\omega=e^{\alpha t}\omega$.  It may happen that $X$ is not complete. The number $a$ is called the {\sl conformality rate}.
\end{itemize}
\end{definition}
By potentially reversing time, we can assume that $a\in (0,1)$ and $\alpha<0$. Thereafter, we will always assume this.
\begin{exercise}\label{ExoLibermann} Assume that $(M^{(2d)}, \omega)$ is a  $2d$-dimensional symplectic manifold with $d\geq 2$. Let $f: M\to M$  be a diffeomorphism and let $a:M\to\mathbb R$ be a function such that $f^*\omega=a\omega$. Prove that the function $a$ is locally constant (hence constant when $M$ is connected). This result is due to Libermann, \cite{Libermann1959}.

\end{exercise}

\begin{remark}\phantom{fish}
\begin{itemize}
\item[$\bullet$]  If $f^*\omega =a\omega$ with $a>0$ and $a\neq 1$, then $f^*\omega^{\wedge d}=a^n\omega^{\wedge d}$, the volume form is multiplied by a constant that is not $1$, hence the manifold has infinite volume. In particular, the manifold is not compact.
\item[$\bullet$] Even if $\omega$ is contracted, the diffeomorphism $f$ is not necessarily a contraction, see e.g. $f: \mathbb R^2\to \mathbb R^2$ defined by $f(x, y)=(x+1, ay)$. Note also that this example does not include any compact orbits.

\end{itemize}
\end{remark}

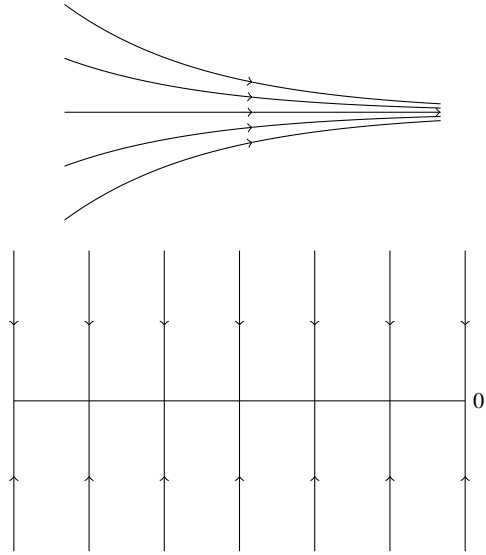
\begin{figure}[h!]
\centering
\begin{tikzpicture}
        \draw[->]  (-2.5,0) -- (0,0) ;
          \draw[->]  (0,0) -- (2.5,0) ;
        \draw[->][domain=-2.5:0][samples=200] plot (\x, {0.2*(0.6^(\x))});
         \draw[domain=0:2.5][samples=200] plot (\x, {0.2*(0.6^(\x))});
        \draw[->][domain=-2.5:0][samples=200] plot (\x, -{0.2*(0.6^(\x))});
         \draw[domain=0:2.5][samples=200] plot (\x, -{0.2*(0.6^(\x))});
        \draw[->][domain=-2.5:0][samples=200] plot (\x, {0.4*(0.6^(\x))});
                \draw[domain=0:2.5][samples=200] plot (\x, {0.4*(0.6^(\x))});
         \draw[->][domain=-2.5:0][samples=200] plot (\x, -{0.4*(0.6^(\x))});
                  \draw[domain=0:2.5][samples=200] plot (\x, -{0.4*(0.6^(\x))});
    \end{tikzpicture}
    \end{figure}
 
\begin{example}\label{E5} We consider on $T^*\mathbb T=\mathbb T\times \mathbb R$ the diffeomorphism $f$ defined by $f(\theta,r)=(\theta, ar)$. There is global attractor (see Definition \ref{Dattr+globattr}, Section \ref{sec:3}), namely $\mathbb T\times \{ 0\}$.

\begin{figure}[h!]
\centering
\begin{tikzpicture}
        \draw  (-3,0) -- (3,0) node[right] {$0$};;
        \draw[->] (-3,-2) -- (-3,-1) ;
         \draw  (-3,-1) -- (-3,0) ;
         \draw[->] (-3,2) -- (-3,1) ;
         \draw  (-3,1) -- (-3,0) ;
         \draw[->] (-2,-2) -- (-2,-1) ;
         \draw  (-2,-1) -- (-2,0) ;
         \draw[->] (-2,2) -- (-2,1) ;
         \draw  (-2,1) -- (-2,0) ;
 \draw[->] (-1,-2) -- (-1,-1) ;
         \draw  (-1,-1) -- (-1,0) ;
         \draw[->] (-1,2) -- (-1,1) ;
         \draw  (-1,1) -- (-1,0) ;

          \draw[->] (0,2) -- (0,1) ;
          \draw  (0,1) -- (0,0) ;
     \draw[->] (0,-2) -- (0,-1) ;
       \draw  (0,-1) -- (0,0) ;
       \draw[->] (3,-2) -- (3,-1) ;
         \draw  (3,-1) -- (3,0) ;
         \draw[->] (3,2) -- (3,1) ;
         \draw  (3,1) -- (3,0) ;
 \draw[->] (2,-2) -- (2,-1) ;
         \draw  (2,-1) -- (2,0) ;
         \draw[->] (2,2) -- (2,1) ;
         \draw  (2,1) -- (2,0) ;
 \draw[->] (1,-2) -- (1,-1) ;
         \draw  (1,-1) -- (1,0) ;
         \draw[->] (1,2) -- (1,1) ;
         \draw  (1,1) -- (1,0) ;
           \end{tikzpicture}
    \end{figure}
\end{example}

\begin{remark}\phantom{fish}
\begin{itemize}
\item[$\bullet$] Let $X$ be a conformally symplectic vector field  defined on  the symplectic manifold $(M, \omega)$. Then $d(i_X\omega)=\alpha \omega$ and the manifold is exact symplectic.
\item[$\bullet$] There exist conformally symplectic diffeomorphisms that are defined on a symplectic manifold that is not exact. Such an example is given in Proposition 2 of \cite{ArnaudFejoz2024}. 
\item[$\bullet$]  When $\omega=-d\lambda$ is exact, note that a vector field $X$ is conformally symplectic with conformality rate $\alpha$ if and only if $d(i_X\omega+\alpha \lambda)=0$.
\end{itemize}
\end{remark}
Hence let us look to the case when $i_X\omega+\alpha \lambda$ is exact.

\begin{definition} Assume $(M, \omega=-d\lambda)$ is exact symplectic, $\alpha\neq 0$ and $H:M\to\mathbb R$ is a $C^{1, 1}$ function. Then the {\sl conformally Hamiltonian} (CH) vector field $X=X_H^\alpha$ is defined by 
$$i_X\omega =\alpha \lambda +dH.$$

\end{definition}

\begin{exercise} Let $H:\mathbb T\times \mathbb R\to \mathbb R$ be defined by $H(\theta, r)=r^2\sin (2\pi \theta)$. Prove that for every $\alpha\neq 0$, the vector field $X_H^\alpha$ is not complete.
\end{exercise}

\begin{example} When $H=0$ and $\alpha=1$, the vector field is called the {\sl Liouville vector field} and is denoted by $Z_\lambda$. We have $i_{Z_\lambda}\omega=\lambda$. On a cotangent bundle endowed with its tautological 1-form, the Liouville flow is given by the equality $\varphi_t(q, p)=(q, e^{-t}p)$ and the zero-section is the global attractor (see Definition \ref{Dattr+globattr}, Section \ref{sec:3}).
\end{example}

\begin{remark}. Let $(M, \omega=-d\lambda)$ be exact symplectic. For $H:M\to\mathbb R$ $C^{1, 1}$, the notation $X_H$ is for the usual (symplectic) Hamiltonian vector field, i.e. $i_{X_H}\omega=dH$. Then we have~: $X_H^\alpha=\alpha Z_\lambda+X_H$.
\end{remark}
\begin{example}\label{E10}  Let $H:\mathbb T\times \mathbb R\to \mathbb R$ be defined by $H(\theta, r)=r\sin (2\pi \theta)$ and let $\alpha\in (0, 2\pi)$. The conformally Hamilton equations are 
\[ \begin{cases} \dot \theta  =\sin 2\pi \theta;\\
\dot r=-(\alpha +2\pi \cos 2\pi \theta )r
\end{cases}
\]
Because there is one point whose $\omega$-limit set is empty, there is no compact global attracteur.

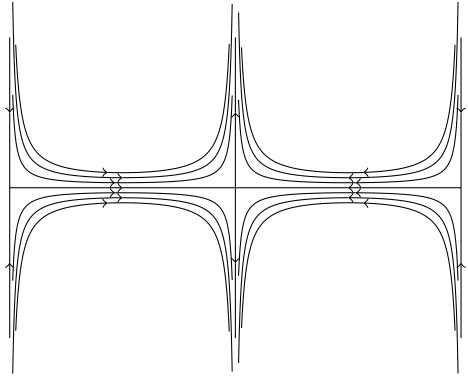
\begin{figure}[h!]
\centering
\begin{tikzpicture}
  \draw  (1.5,0) -- (0,0) ;
 \draw[->]  (3,0) -- (1.5,0) ;
    \draw (-1.5,0) -- (0,0) ;
     \draw[->]  (-3,0) -- (-1.5,0) ;
       
        \draw[->] (-3,-2) -- (-3,-1) ;
         \draw  (-3,-1) -- (-3,0) ;
         \draw[->] (-3,2) -- (-3,1) ;
         \draw  (-3,1) -- (-3,0) ;
          (-1,0) ;

          \draw[->] (0,0) -- (0,1) ;
          \draw  (0,1) -- (0,2) ;
     \draw(0,-1) -- (0,-2) ;
       \draw [->]  (0,0) -- (0,-1) ;
       \draw[->] (3,-2) -- (3,-1) ;
         \draw  (3,-1) -- (3,0) ;
         \draw[->] (3,2) -- (3,1) ;
         \draw  (3,1) -- (3,0) ;
  
  \draw[->][domain=-2.96:-1.5][samples=200] plot (\x, {0.3*((\x+0.001)^(-1)*(\x+3.001)^(-1)});
  \draw[domain=-1.5:-0.04][samples=200] plot (\x, {0.3*((\x+0.001)^(-1)*(\x+3.001)^(-1)});
  
  \draw[->][domain=-2.96:-1.6][samples=200] plot (\x, {0.15*((\x+0.001)^(-1)*(\x+3.001)^(-1)});
  \draw[domain=-1.6:-0.04][samples=200] plot (\x, {0.15*((\x+0.001)^(-1)*(\x+3.001)^(-1)});
  
  \draw[->][domain=-2.92:-1.7][samples=200] plot (\x, {0.45*((\x+0.001)^(-1)*(\x+3.001)^(-1)});
  \draw[domain=-1.7:-0.08][samples=200] plot (\x, {0.45*((\x+0.001)^(-1)*(\x+3.001)^(-1)});

 \draw[->][domain=-2.96:-1.5][samples=200] plot (\x, {-0.3*((\x+0.001)^(-1)*(\x+3.001)^(-1)});
  \draw[domain=-1.5:-0.04][samples=200] plot (\x, {-0.3*((\x+0.001)^(-1)*(\x+3.001)^(-1)});
  
  \draw[->][domain=-2.96:-1.6][samples=200] plot (\x, {-0.15*((\x+0.001)^(-1)*(\x+3.001)^(-1)});
  \draw[domain=-1.6:-0.04][samples=200] plot (\x, {-0.15*((\x+0.001)^(-1)*(\x+3.001)^(-1)});
  
  \draw[->][domain=-2.92:-1.7][samples=200] plot (\x, {-0.45*((\x+0.001)^(-1)*(\x+3.001)^(-1)});
  \draw[domain=-1.7:-0.08][samples=200] plot (\x, {-0.45*((\x+0.001)^(-1)*(\x+3.001)^(-1)});

\draw[->][domain=2.96:1.5][samples=200] plot (\x, {0.3*((\x+0.001)^(-1)*(\x-3.001)^(-1)});
  \draw[domain=1.5:0.04][samples=200] plot (\x, {0.3*((\x+0.001)^(-1)*(\x-3.001)^(-1)});
  
  \draw[->][domain=2.96:1.6][samples=200] plot (\x, {0.15*((\x+0.001)^(-1)*(\x-3.001)^(-1)});
  \draw[domain=1.6:0.04][samples=200] plot (\x, {0.15*((\x+0.001)^(-1)*(\x-3.001)^(-1)});
  
 \draw[->][domain=2.92:1.7][samples=200] plot (\x, {0.45*((\x+0.001)^(-1)*(\x-3.001)^(-1)});
  \draw[domain=1.7:0.08][samples=200] plot (\x, {0.45*((\x+0.001)^(-1)*(\x-3.001)^(-1)});

\draw[->][domain=2.96:1.5][samples=200] plot (\x, {-0.3*((\x+0.001)^(-1)*(\x-3.001)^(-1)});
  \draw[domain=1.5:0.04][samples=200] plot (\x, {-0.3*((\x+0.001)^(-1)*(\x-3.001)^(-1)});
  
  \draw[->][domain=2.96:1.6][samples=200] plot (\x, {-0.15*((\x+0.001)^(-1)*(\x-3.001)^(-1)});
  \draw[domain=1.6:0.04][samples=200] plot (\x, {-0.15*((\x+0.001)^(-1)*(\x-3.001)^(-1)});
  
 \draw[->][domain=2.92:1.7][samples=200] plot (\x, {-0.45*((\x+0.001)^(-1)*(\x-3.001)^(-1)});
  \draw[domain=1.7:0.08][samples=200] plot (\x, {-0.45*((\x+0.001)^(-1)*(\x-3.001)^(-1)});

                           \end{tikzpicture}
    \end{figure}

\end{example}
We recall
\begin{definition}
Let $X$ be a topological space and let $f: X\to X$ be a homeomorphism. Let $x\in X$. Then 
\begin{itemize}
\item[$\bullet$] A point $y\in X$ is in the $\alpha$-limit set $\alpha(x)$ of $x$ if it is a limit point of $(f^{-n}(x))_{n\in\mathbb N}$.
\item[$\bullet$]A point $y\in X$ is in the $\omega$-limit set $\omega(x)$ of $x$ if it is a limit point of $(f^n(x))_{n\in\mathbb N}$.

\end{itemize} 
\end{definition}
For dynamicists, the interesting part of the dynamics is contained in the {\sl non-wandering set}.

\begin{definition} Let $X$ be a topological space and let $f: X\to X$ be a homeomorphism.
\begin{itemize}
\item[$\bullet$] a set $U\subset X$ is {\sl wandering} if $\forall n\in\mathbb N^*, f^n(U)\cap U=\emptyset$;
\item[$\bullet$] a point $x\in X$ is {\sl wandering} if $x$ has a wandering neighbourhood;
\item[$\bullet$] a point $x\in X$ is positively (resp. negatively) {\sl recurrent} if there exists an increasing sequence $(n_k)_{k\in\mathbb N}\in \mathbb N^\mathbb N$ such that 
$$\lim_{k\to +\infty}f^{n_k}(x)=x \text{ (resp. }\lim_{k\to +\infty}f^{-n_k}(x)=x \text{)}.$$
\end{itemize}
\end{definition}
The set of non-wandering points is the {\sl non-wandering set} and is denoted by $\Omega(f)$. The set of positively (resp. negatively) recurrent points is denoted $\mathcal R_+(f)$ (resp. $\mathcal R_-(f)$). Recurrent points belong to $\Omega (f)$. The set $\mathcal P(f)$ of periodic points is contained in $\mathcal R_-(f)\cap\mathcal R_+(f)$.\\
The four  sets $\mathcal P(f)$, $\Omega(f)$, $\mathcal R_+(f)$ and $\mathcal R_-(f)$ are $f$ invariant. Moreover, $\Omega(f)$ is closed.
\begin{exercise} Find an example where $\mathcal R_+(f)\neq \Omega(f)$.
\end{exercise}
Definitions exist for flows also. We will not  give them.
\begin{question}
Let $f$ be a conformally symplectic diffeomorphism of a symplectic manifold. Has $\Omega(f)$ zero volume? No interior?
\end{question}
We do not know the answer in the general case. The answer is positive when $\Omega(f)$ has finite volume because we have~: 
\begin{lemma}
Let $f: M\to M$ be a conformally symplectic diffeomorphism of a symplectic manifold. Then, every measurable invariant subset has zero or infinite volume.
\end{lemma} 
In particular, when $\Omega(f)$ is compact, it has zero volume and no interior.
\begin{definition} Let $X$ be a topological space. Let $(x_n)_{n\in\mathbb N}$ be a sequence of points of $X$. The sequence {\sl tends to infinity} if for every compact subset $K$ of $X$, there exists $N\in\mathbb N$ such that 
$$\forall n\geq N; x_n\notin K.$$\end{definition}
\noindent {\bf Notation}
Let $X$ be a topological space and let $f: X\to X$ be a homeomorphism. We denote by $B(f)$ the set of $x\in X$ such that the sequence $(f^{-n}(x))_{n\in\mathbb N}$ does not  tend to infinity when $n\to +\infty$.\\
In other words, $B(f)$ is the set of points whose negative orbit does not tend to infinity. The complement of $B(f)$ is the unstable set of infinity, i.e. the set of points whose negative orbit tends to infinity.

\begin{proposition}\label{PunstablesetInfinity}
Let $M$ be a   Riemannian manifold endowed with the corresponding volume form $V$ and let $f: M\to M$ be a diffeomorphism. We assume that there exists $a\in (0, 1)$ such that the function $A:M\to \mathbb R$ defined by $f^*V=AV$ satisfies $\vert A\vert \leq a$. 
 Then $V(B(f))=0$. 
\end{proposition}
This means that Lebesgue almost every point in $M$ comes from infinity (when the time goes to $-\infty$). In other words, the unstable set of infinity has full Lebesgue measure.

Of course, conformally dynamics satisfy the hypotheses of Proposition \ref{PunstablesetInfinity}.

\begin{corollary}\label{CunstablesetInfinity}
Let $(M^{(2d)},\omega)$ be a symplectic manifold and let $f: M\to M$ be a diffeomorphism with conformality ratio $a\in (0, 1)$. Then $\omega^{\wedge d}(B(f))=0$. 
\end{corollary}

Also, the same result is true for flows. See \cite{ArnaudSuZavidovique2015} for a proof.

\begin{corollary}
Let $M$ be a   Riemannian manifold endowed with the corresponding volume form $V$ and let $f: M\to M$ be a diffeomorphism. We assume that there exists $a\in (0, 1)$ such that the function $A:M\to \mathbb R$ defined by $f^*V=AV$ satisfies $\vert A\vert \leq a$. 
Then $V(\mathcal R_-(f))=0$. 
\end{corollary}

\begin{corollary}
Let $(M^{(2d)},\omega)$ be a symplectic manifold and let $f: M\to M$ be  a conformally symplectic diffeomorphism with conformality ratio $a\in (0, 1)$. Then $\omega^{\wedge d}(\mathcal R_-(f))=0$. 
\end{corollary}
\begin{corollary}
Let $M$ be a   Riemannian manifold endowed with the corresponding volume form $V$ and let $f: M\to M$ be a diffeomorphism. We assume that there exists $a\in (0, 1)$ such that the function $A:M\to \mathbb R$ defined by $f^*V=AV$ satisfies $\vert A\vert \leq a$. Then the set of points whose negative orbit is relatively compact has zero Lebesgue measure.
\end{corollary}

\begin{corollary}
Let $(M,\omega)$ be a symplectic manifold and let $f: M\to M$ be a diffeomorphism with conformality ratio $a\in (0, 1)$. Then the set of points whose negative orbit is relatively compact has zero Lebesgue measure.
\end{corollary}
\begin{proof}[Proof of Proposition \ref{PunstablesetInfinity}]
We write $M=\bigcup_{n\in \mathbb N}K_n$ where $(K_n)_{n\in\mathbb N}$ is an increasing sequence of compact subsets of $M$ and define 
$B_N(f)=\{ x\in K_N;$
$$ \exists(k_n)_{n\in\mathbb N}\text{ increasing sequence in }\mathbb N; \forall n\in\mathbb N, f^{-k_n}(x)\in K_N\}.$$
Then the first return map $\mathcal P_N: B_N\to B_N$ for $f^{-1}$ is defined from $B_N(f)$ to $B_N(f)$ and for every measurable subset $A$ of $B_N$, we have $V(\mathcal P_N(A))\geq \frac{1}{a}V(A)$ where $\frac{1}{a}>1$.  As $B_N(f)\subset K_N$ has finite volume and $\mathcal P_N(B_N(f))\subset B_N(f)$, this implies that $V(B_N(f))\geq \frac{1}{a}V(B_N(f))$  and $B_N(f)$ has finite volume, hence
 $V(B_N(f))=0$. As $B(f)=\bigcup_{n\in\mathbb N} B_n(f)$, we deduce that\\
  $V(B(f))=0$.
\end{proof}
\begin{exercise}
Prove  the conclusion of Proposition  \ref{PunstablesetInfinity} if we just assume that for every measurable subset $A$ of $M$ with finite and non zero volume, we have $V(f(A))<V(A)$.
\end{exercise}
If we look to our previous examples, we have
\begin{itemize}
\item[$\bullet$] For $f: \mathbb R^2\to \mathbb R^2$ defined by $f(x, y)=(x+1, ay)$, $B(f)=\emptyset$;
\item[$\bullet$] for $f: \mathbb T\times \mathbb R\to \mathbb T\times \mathbb R$ defined by $f(\theta,r)=(\theta, ar)$ (Example \ref{E5}), $B(f)=\mathbb T\times \{ 0\}$;
\item[$\bullet$] for  $H:\mathbb T\times \mathbb R\to \mathbb R$  defined by $H(\theta, r)=r\sin (2\pi \theta)$ and   $\alpha\in (0, 2\pi)$ (see Example \ref{E10}), let $f$ be the time-one map of the conformally Hamiltonian flow of $H$. Then $$B(f)=(\mathbb T\times \{0\})\cup (\{\frac{1}{2}\}\times \mathbb R)$$ is the union of the zero-section and the unstable sub-manifold of the saddle fixed point $(\frac{1}{2}, 0)$.
\end{itemize}
We conclude this section with a few remarks concerning the comparison between two concepts: conformally Hamiltonian vector fields and conformally symplectic vector fields on an exact symplectic manifold $(M, \omega=-d\lambda)$. 

Let us assume that $X$ is a  conformally symplectic vector field with conformality rate $\alpha$. Then $L_X\omega=\alpha\omega$, i.e. $d(i_X\omega +\alpha \lambda)=0$. Then the $1$-form $\eta =\frac{1}{\alpha}(i_X\omega+\lambda)$ is closed and $\tilde \lambda=\lambda -\eta$ is a primitive of $\omega$. We have
$$i_X\omega+\alpha\tilde \lambda=0.$$
Hence every conformally symplectic vector field is conformally Hamiltonian, and even the Liouville vector field for some primitive of $\omega$.

In Appendix B of \cite{ArnaudFejoz2024}, the following result is proven.
\begin{proposition}[Arnaud-F\'ejoz, \cite{ArnaudFejoz2024}] 
Let $X$ be a conformally symplectic vector field. Assume that the vector field $Y$ defined by
$$i_Y\omega=\frac{1}{\alpha}i_X\omega+\lambda$$ is complete. Then there exists a symplectically isotopic to identity diffeomorphism $g: M\to M$  such that $g^*X$ is conformally Hamiltonian for the same Liouville form $\lambda$.
\end{proposition}

There exist conformally symplectic diffeomorphisms that are defined on a symplectic manifold that is not exact. Such an example is given in Proposition 2 of \cite{ArnaudFejoz2024}. 

The manifold $M$ in question is a 6-dimensional submanifold of the 8-dimensional manifold $T^*\mathbb T^4=\mathbb T^4\times \mathbb R^4$. If $p=\frac{\sqrt{5}-1}{2}$, then 
$$M=\{ (\theta_1, \theta_2, \theta_3, \theta_4, r_1, r_2, r_3, r_4)\in \mathbb T^4\times \mathbb R^4; r_2=pr_1\text{ and }r_4=pr_3\}.$$
We denote by $\Omega_1$ the 2-form $(d\theta_2-pd\theta_1)\wedge (d\theta_4-pd\theta_3)$ and by $\Omega_2$ the 2-form induced by the usual symplectic form of $T^*\mathbb T^4$ on $M$. Then $\Omega=\Omega_1+\Omega_2$ is a non-exact symplectic form on $M$. The map $F:M\to M$ defined by
$$F(\theta, r)=(2\theta_1+\theta_2, \theta_1+\theta_2, 2\theta_3+\theta_4, \theta_3+\theta_4, (3-\sqrt{5})/2.r)$$
is conformally symplectic for $\Omega$ with conformality ratio equal to $(7-3\sqrt{5})/2$.

\section{Isotropy}\label{sec:2}
In this section, we explain that the $\omega$-isotropy is preserved by any conformally symplectic dynamics. We then give various conditions that imply that a given invariant submanifold is isotropic. Particular attention is paid to the links between isotropy and
topological entropy.  We conclude by proving that any closed Lagrangian submanifold invariant
under a conformal Hamiltonian flow is exact Lagrangian. Note that Example \ref {ExArnFejnonisot} in section \ref{sec:3} is an example of  a non-isotropic invariant submanifold.

Throughout this section, we assume that $(M^{(2d)}, \omega)$ is a $2d$-dimensional symplectic manifold.
\begin{definition}
Let $L$ be a submanifold of $M$. Then
\begin{itemize}
\item $L$ is {\sl isotropic} if 
$$\forall x\in L, \forall u, v\in T_xL, \omega(u, v)=0.$$
\item $L$ is {\sl Lagrangian} if $L$ is isotropic with maximal dimension, i.e. $\dim L=d$.
\item when $\omega_{|TL}$ has constant rank, its kernel is integrable and the resulting (isotropic) foliation is called the {\sl characteristic foliation}.
\end{itemize}
\end{definition}
\begin{example}\phantom{fish}
\begin{itemize}
\item A curve is always isotropic. Moreover, when $d=1$, it is Lagrangian;
\item If $\eta$ is a 1-form on the closed manifold $N$, then the graph of $\eta$ is a Lagrangian submanifold of $T^*N$ if and only if $\eta$ is closed.
\end{itemize}
\end{example}

\begin{remark}
When $f:M\to M$ is conformally symplectic, then we have
$$\forall x\in M, \forall u, v\in T_xM, \omega(u,v)=0\Longrightarrow \omega(Dfu, Dfv)=0.$$
Hence a conformally symplectic dynamics 
\begin{itemize}
\item preserves $\omega$-orthogonality;
\item  maps an isotropic submanifold on an isotropic submanifol;
\item  maps a Lagrangian submanifold on a Lagrangian submanifold;
\item maps the characteristic foliation $\mathcal F$ of a submanifold $N$ on the characteristic foliation $f(\mathcal F)$ of $f(N)$.
\end{itemize}
\end{remark}
\begin{exercise}
Assume  that $d\geq 2$ and that $f:M\to M$ is a diffeomorphism. Then, if $f$ preserves $\omega$-orthogonality, $f$ is conformally symplectic (see Proposition 1.1 of \cite{LiveraniWojtkowski1998}).
\end{exercise}

In the symplectic setting, several results provide dynamical conditions implying that a submanifold is isotropic or Lagrangian. For example, we know that a stable or unstable (immersed) submanifold of a hyperbolic periodic orbit is Lagrangian. Herman also proved that, under certain conditions of exactness, a KAM torus is always Lagrangian, see 3.2. in \cite{Herman1989}.

Based on the examples in section \ref{sec:1}, we see that the zero-section of $T^*N$, which is invariant under the Liouville flow, is an invariant Lagrangian submanifold, which is also an attractor (see the next section for the definition of an attractor).

\begin{proposition}\label{Pstableunstable}
Assume that $f:M\to M$ is conformally symplectic such that $f^*\omega=a\omega$ with $a\in (0, 1)$. Let $x$ be a periodic point with period $m\geq 1$. Let us introduce the notations
$$\Lambda_-=\{ \lambda\in\text{Spec}Df^m(x); |\lambda|<\sqrt{a^m}\} \text{ and } \Lambda_+=\{ \lambda\in \text{Spec}(Df^m(x)); |\lambda|>1\}.$$
We choose $b, c\in \mathbb R$ such that $\sup|\Lambda_-|<b<\sqrt{a^m}<1<c<\inf|\Lambda_+|$. Then $$W^-=\{ y\in M; \lim_{n\to\infty} \frac{1}{b^n}d(f^nx, f^ny)=0\}$$ and $$W^+=\{ y\in M; \lim_{n\to\infty}c^nd(f^{-n}y, f^{-n}x)=0\}$$ are isotropic immersed submanifold of $M$ that are invariant by $f^m$.
\end{proposition}

\begin{proof}[Proof of Proposition \ref{Pstableunstable}] Let $E_\pm$ be the sum of the characteristic subspaces of $T_x(T^*M)$ corresponding to the eigenvalues of $\Lambda_\pm$. It is classical that $W^\pm$ is an immersed submanifold such that $T_xW^\pm=E_\pm$ see Exercise 6.11 of \cite{Irwin2001}. 
We prove the isotropy of $W^-$.

Let $y\in W^-$ and $u, v\in T_yW^-$. Then the sequence $(f^{nm}y)_{n\in\mathbb N}$ converges to $x$ and $\omega(u, v)=\frac{1}{a^{nm}}\omega(Df^{nm}u, Df^{nm}v)$ where 
$$Df^{nm}(y)=Df^m(f^{(n-1)m}y)\dots Df^m(f^my)Df^m(y).$$
As $(f^{nm}y)_{n\in\mathbb N}$ converges to $x$  and $\| Df^m(x)_{|E_-}\|<b<\sqrt{a^m}$, then
$$\omega(u, v)=\vert \frac{1}{a^{nm}}\omega(Df^{nm}u, Df^{nm}v)\vert\leq \text{Constant}.(\frac{b^{2}}{a^{m}})^n\| u\|.\| v\|$$
where $\lim_{n\to\infty}(\frac{b^{2}}{a^{m}})^n=0$. We deduce that $W^-$ is isotropic.

\end{proof}

\begin{remark}
For a flow, we need to add the direction of the vector field in $E_\pm$. The resulting submanifolds are also isotropic.
\end{remark}

\begin{proposition}[Arnaud-Fejoz] \label{Pdegeneratesubmanifold}  Assume that $f:M\to M$ is such that $f^*\omega=a\omega$ with $a\in (0, 1)$. Let $N\subset M$ be a closed $f$-invariant submanifold. Then $\omega_{|TN}$ is degenerate at every point.
\end{proposition}

\begin{corollary} Assume that $f:M\to M$ is such that $f^*\omega=a\omega$ with $a\in (0, 1)$. Let $S\subset M$ be a closed $f$-invariant surface. Then $S$ is isotropic.
\end{corollary}

\begin{proof}[Proof of Proposition \ref{Pdegeneratesubmanifold} ] When $\dim N$ is odd, there is nothing to prove. If $\dim N=2k$, assume that there is some open subset $U$ of $N$ on which $\omega^{\wedge k}>0$. Then $\int_U\omega^{\wedge k}>0$. and 
$$\int_{f^{-n}U}\omega^{\wedge k}=\frac{1}{a^n}\int_U\omega^{\wedge k} \stackbin[n\to \infty]{\longrightarrow}{}+\infty.$$
But we have 
$$\vert \int_{f^{-n}U}\omega^{\wedge k}\vert\leq \int_N\| \omega^{\wedge k}\|<+\infty,$$
thus a contradiction.

\end{proof}
\begin{proposition}[Calleja, Celletti, de la Llave, \cite{CallejaCellettiLlave2013}]\label{PCCdL2013}
Assume that $f:M\to M$ is such that $f^*\omega=a\omega$ with $a\in (0, 1)$. Then, if $j: \mathbb T^m\hookrightarrow M$ is a $C^1$ embedded torus that is $f$-invariant and such that $f_{|j(\mathbb T^m)}$ is $C^1$ conjugate to a rotation, then $j(\mathbb T^m)$ is isotropic.
\end{proposition}
\begin{proof}[Proof of Proposition \ref{PCCdL2013}] We use the notations $\theta=(\theta_1, \dots, \theta_m)\in \mathbb T^m$ and $R_\beta:\mathbb T^m\to \mathbb T^m$ such that $R_\beta(\theta)=\theta+\beta$.

We introduce 
$$\omega_0=j^*\omega=\sum_{i<j}b_{i, j}(\theta)d\theta_i\wedge d\theta_j.$$ Then we have $j^{-1}\circ f\circ j=R_\beta$ so 
$$\sum_{i<j} b_{i, j}(R_\beta\theta)d\theta_i\wedge d\theta_j=R_\beta^*\omega_0=a\omega_0=\sum_{i<j}ab_{i, j}(\theta)d\theta_i\wedge d\theta_j.$$
Hence $$\forall i, j , \forall \theta, b_{i, j}(R_\beta \theta)=ab_{i, j}(\theta),$$
which implies that every function $b_{i, j}$ is the zero function.

\end{proof}
The concept of {\sl topological entropy} was introduced by Adler, Konheim \& McAndrew, \cite{AdlerKonheimMcAndrew1965}. This concept measures the degree to which dynamics are chaotic.
For any homeomorphism $h$ of a closed manifold $N$, its {topological entropy} is denoted  $\text{Ent}(h)$.  We will not give a precise definition, but simply recall that if $(U_i)_{1\leq i\leq m}$ is an open covering  of $N$, then each piece of orbit $(h^ix)_{0\leq i\leq n}$ has at least one $(n+1)$-itinerary $(n_i(x))_{0\leq i \leq n}\in [0, m]^{n+1}$ such that $h^i(x)\in U_{n_i(x)}$. If $N_{N+1}$ is the minimal number of possible $(n+1)$-itineraries for all points, then $\lim_{n\to \infty} \frac{1}{n}\log N_n$ exists and is at most equal to $\text{Ent}(h)$. Furthermore,  recall that $\text{Ent}(h)=\text{Ent}(h^{-1})$

Yomdin theory, \cite{Yomdin1987,Gromov1987}, etablishes  a link between the topological entropy of a dynamics and the expansion of the Riemannian volume of the images of the submanifolds. The Riemannian volume can be compared to the symplectic form. There is therefore a direct consequence of Yomdin theory.

\begin{proposition}[Arnaud-F\'ejoz, \cite{ArnaudFejoz2024}]  Assume that $f:M\to M$ is a diffeomorphism such that $f^*\omega=a\omega$ with $a\in (0, 1)$. Let $L\subset M$ be a $f$-invariant $C^\infty$ closed submanifold such that $\text{Ent}(f_{|L})<|\log a|$. Then $L$ is isotropic.

\end{proposition}

Therefore,  a small disorder (small topological entropy) of $f_{|L}$ implies the isotropy of $L$.

\begin{corollary}
With the same assumptions, when $\text{Ent}(f_{|L})=0$, for example when $f_{|L}$ is $C^0$ conjugate to a rotation, then $L$ is isotropic.
\end{corollary}

Without resorting to Yomdin theory and by weakening the hypotheses concerning the smoothness of the invariant submanifold, we have the following theorem.
\begin{theorem}[Arnaud-Fejoz,  \cite{ArnaudFejoz2024}]\label{Tentconstrank} Let $N^{(n)}$ be a closed Riemannian $C^2$ manifold and\begin{itemize}
\item let $\mathcal F$ be a $C^2$ foliation induced by a subbundle $F$ of $TN$ of rank $p\leq n-1$;
\item let $\Omega$ be a $(n-p)$-form on $N$ which induces a volume form on submanifolds that are transversal to $\mathcal F$;
\item let $f$ be a $C^2$ diffeomorphism of $N$ preserving $\mathcal F$ such that $f^*\Omega=b\Omega$ for some $b>1$.
\end{itemize}
Then $\text{Ent}(f)\geq \log b$.
\end{theorem}

\begin{corollary}[Arnaud-Fejoz,  \cite{ArnaudFejoz2024}]\label{Centconstrank} Let $f:M\to M$ be a $C^3$ conformally symplectic diffeomorphism such that $f^*\omega=a\omega$ with $a\in (0, 1)$. Assume that $N$ is a $C^3$ invariant submanifold such that the induced form $\omega_{|TN}$ has constant rank. Then
$$\text{Ent}(f_{|N})\geq \frac{\text{rank} (\omega_{|TN})}{2}|\log a|.$$
In particular, if the entropy of $f_{\vert N}$ vanishes, then $N$ is isotropic.

\end{corollary} 
\begin{proof}[Proof of Corollary \ref{Centconstrank}] We use Theorem \ref{Tentconstrank} to $f_{|N}^{-1}$, with $\mathcal F$ being the characteristic foliation of $\omega_{|TN}$ with rank denoted by $p$, $\Omega=\omega^{\wedge k}$ where $k=\frac{\dim N-p}{2}$ and $b=\frac{1}{a^k}$.

\end{proof}
\begin{remark}
Hypotheses of Corollary \ref{Centconstrank}  are satisfied when $N$ is a $C^3$ invariant submanifold such that $f_{\vert N}$ is minimal (every orbit is dense), in particular, when $f_{|N}$ is $C^0$ conjugate to an ergodic rotation.
\end{remark}
\begin{proof}[Sketch of the proof of Theorem \ref{Tentconstrank}]
\begin{itemize} 
\item We interpret $\Omega$ as a measure  on the pieces of submanifolds that are $(n-p)$-dimensional and transversal to $\mathcal F$.
\item We construct a partition of $N$ into submanifolds with boundary $Q_1, \dots, Q_J$ and we want to estimate the number $N_k$ of itineraries of length $k$ because we know that $\text{Ent}(f)\geq \lim_{n\to \infty}\frac{1}{k}\log (N_k)$.
\item We fix  $N\subset Q_1$ which is a piece of submanifold that is $(n-p)$-dimensional and transversal to $\mathcal F$, and we show that the measure of the image by $f^k$ of the set of  points of $N$ with a fixed $k$-itinerary is less than a certain constant $C$ (that is independent of $k$).
\end{itemize}
We deduce that $b^k\Omega(N)=\Omega(f^kN)\leq N_kC$, hence 
$$\text{Ent}(f)\geq \lim_{k\to\infty} \frac{1}{k}\log (N_k)\geq \lim_{k\to\infty}(\log b+\frac{\log\Omega(N)-\log C}{k})=\log b.$$

\end{proof}
To conclude this section, let us explain why a Lagrangian submanifold that is invariant under
a conformally Hamiltonian flow is more than Lagrangian: it is exact Lagrangian. We
recall
\begin{definition} Let $(M^{(2d)}, \omega=-d\lambda)$ be a $2d$-dimensional  exact symplectic manifold. Let $L^{(d)}\subset M$ be a $d$-dimensional submanifold of $M$. Then $L$ is {\sl exact Lagrangian} if $\lambda_{|L}$ is exact.
\end{definition}
Note that an exact Lagrangian submanifold is always Lagrangian. \\
In the particular case of the cotangent bundle $T^*N$, the graph of a $1$-form $\eta$ on $N$ is exact Lagrangian if and only if the 1-form $\eta$ is exact, i.e. there exists $u:N\to\mathbb R$ such that $\eta=du$.
\begin{proposition}[Arnaud-Fejoz,  \cite{ArnaudFejoz2024}]\label{PinvLinvEL} Let $(M, \omega=-d\lambda)$ be an exact symplectic manifold and let $X$ be a conformally Hamiltonian vector field. Then every flow invariant closed Lagrangian submanifold is exact Lagrangian.
\end{proposition}
\begin{proof}[Proof of Proposition \ref{PinvLinvEL}]\phantom{fish}
 \begin{lemma}\label{Lhamexact}
Let $X$ such that $L_X\omega=\alpha\lambda+dH$. Then, if $(\varphi_t)$ is the flow of $X$, the 1-form $$\varphi_t^*\lambda -e^{-\alpha t}\lambda=dS_t$$
is exact.
\end{lemma}
\begin{proof}[Proof of Lemma \ref{Lhamexact}] The derivative of $\eta_t=\varphi_t^*\lambda-e^{-\alpha t}\lambda$ is
\[\begin{split}
\varphi_t^*L_X\lambda+\alpha.e^{-\alpha t}\lambda&= \varphi_t^*(di_X\lambda)+\varphi_t^*(-i_X\omega)+\alpha.e^{-\alpha t}\lambda\\
&=\varphi_t^*(di_X\lambda)-\alpha \varphi_t^*(\lambda)-\varphi_t^*(dH)+\alpha.e^{-\alpha t}\lambda
\end{split}\]

Hence we have
$$\begin{cases}
\frac{d\eta_t}{dt}=-\alpha \eta_t+d(\varphi_t^*(i_X\lambda)-H\circ\varphi_t)\\
\eta_0=0
\end{cases}$$
and we deduce the wanted result for 
$$S_t=\int_0^te^{\alpha(s-t)}(\varphi_s^*(i_X(\lambda))-H\circ \varphi_s)ds.$$
\end{proof}
Denoting the De Rham cohomology class by $[\cdot]$ and using the same notation for the homology, we deduce that $[\varphi_t^*\lambda_{|L}]=e^{-\alpha t}[\lambda_{|L}]$. As the action of the flow on $H_1(L)$ is $Id$, we have for every loop $\gamma$ on $L$
$$\langle [\lambda_{|L}], [\gamma]\rangle=\langle [\lambda_{|L}, \varphi_t\circ\gamma\rangle=\langle [\varphi_t^*\lambda_{|L}], [\gamma]\rangle=e^{-\alpha t}\langle [\lambda_{|L}],[\gamma]\rangle.$$
We deduce that $[\lambda_{|L}]=0$.
\end{proof}

\section{Attractors}\label{sec:3}

In this section, we define the concepts of attractor and global attractor. We give examples of conformally symplectic dynamics that have a global attractor, and others that do not. We explain that the global attractor (when it exists) is unique and give some of its properties. We give a sufficient condition for a conformally Hamiltonian dynamics of a cotangent bundle to have a global attractor. We then  introduce two very classical examples~: Ma$\tilde{\text{n}}$\'e example and damped mechanical systems. These two examples   always have a global attractor, which is sometimes a submanifold and sometimes not. We then conclude with an example  due to Arnaud and Fejoz that has a global attractor that is a non-isotropic submanifold.

Let us now introduce the concept of global attractor, as defined in \cite{ArnaudSuZavidovique2015}. There are  similar definitions for diffeomorphisms.
\begin{definition}\label{Dattr+globattr}
 Let $(\varphi_t)$ be a flow defined for $t\geq 0$ on a manifold $N$, and let $K$ be a non-empty compact subset of $N$.
\begin{itemize}
\item $K$ is an {\sl attractor} if there exists an  open subset $U$ of $N$ such that $K\subset U$ and $\forall t>0, \varphi_t(\overline{U})\subset U$ and  $K=\bigcap_{t\geq 0}\varphi_t(U)$.
\item $K$ is the {\sl global attractor} if it is an attractor and if for every neighbourhood $V$ of $K$ and every $x\in N$, there exists $t\in \mathbb R$ such that $\varphi_t(x)\in V$.
\end{itemize}

\end{definition}
\begin{remark}
Observe that an attractor is always flow-invariant. 
\end{remark}

This concept is equivalent to another that may seem stronger, which we explain in the following proposition.
\begin{proposition}\label{Pattractordefstrongereq}
Let $(\varphi_t)$ be a flow defined for $t\geq 0$ on a manifold $N$, and let $K$ be a non-empty compact subset of $N$.
Then $K$ is an attractor if and only if, for every open subset $U$ of $N$ that contains $K$, there exists an open subset $V$ of $U$ that contains $K$ such that   $\forall t>0, \varphi_t(\overline{V})\subset V$ and  $K=\bigcap_{t\geq 0}\varphi_t(V)$.

\end{proposition}
\begin{proof}[Proof of Proposition \ref{Pattractordefstrongereq}] Let us assume that $K$ is an attractor. Let $W$ be an    open subset  of $N$  such that $K\subset W$ and $\forall t>0, \varphi_t(\overline{W})\subset W$ and  $K=\bigcap_{t\geq 0}\varphi_t(W)$. \\
Let $U$ be an open subset that contains $K$. Even if it means reducing $U$, we can assume that that there exists a relatively compact open subset $U_1$ of $N$ such that $\overline{U}\subset U_1$.
Then the family $(\varphi_t(\overline{W})\cap (\overline{U_1}\backslash U))_{t\geq 0}$ is a decreasing family of compact subsets whose intersection is empty. This implies that there exists $T\geq 0$ such that $\forall t\geq T, \varphi_t(\overline{W})\cap (\overline{U_1}\backslash U)=\emptyset$ hence such that $\varphi_T(\overline{W})\subset U\cup(N\backslash \overline{U_1})$. As $U$ and $N\backslash \overline{U_1}$ are disjoint open subsets, every positive orbit that is contained in $U\cup(N\backslash \overline{U_1})$ is contained either in $U$ or in $ N\backslash \overline{U_1}$. We  then choose $V=\varphi_T(W)\cap U$.
\end{proof}

Albert Fathi explained to us that the definition of an attractor can be weakened as follows.  We would also like to thank Baptiste Serraille for asking the question.  A related reference is \cite{Fathi2022}. 

\begin{proposition}\label{Pattractordeflightereq}
Let $(\varphi_t)$ be a flow defined for $t\geq 0$ on a manifold $N$, and let $K$ be a non-empty  compact subset of $N$.
Then $K$ is an attractor if and only if there exists an open subset $U$ of $N$  such that  $K=\bigcap_{t\geq 0}\varphi_t(U)$.
\end{proposition}
\begin{proof}[Proof of Proposition \ref{Pattractordeflightereq}]
Assume    there exists   an open neighborhood $U$ of $K$ such that  $K=\bigcap\limits_{t\geq0}\varphi_t(U)$. Let $V\subset N$ be an open set such that $K\subset V \subset \overline V \subset U$. Then $K=\bigcap\limits_{t>0}\varphi_t(\overline V)$.  It is proven in \cite[Theorem 1.4]{Fathi2022} that there exists a smooth Lyapounov  function $f : N \to [0,+\infty)$ such that $f^{-1}(\{0\}) = K$ and $Df(x)\cdot X(x) <0$ for all $x \in \overline V\setminus K$ where $X$ denotes the vector field generating the flow. Then setting $U_\varepsilon = f^{-1}([0,\varepsilon))$ for $\varepsilon>0$ sufficiently small   gives an open set that contains $K$ and such that  $\varphi_t(\overline{U_\varepsilon})\subset U_\varepsilon$ and $K=\bigcap\limits_{t>0}\varphi_t(U_\varepsilon)$. 
\end{proof}

\begin{example} Let's return to the examples in Section \ref{sec:1}.
\begin{itemize}
\item For the Liouville flow on $T^*N$, the zero-section is the global attractor.
\item For the diffeomorphism $f: \mathbb R^2\to \mathbb R^2$ defined by $f(x, y)=(x+1, ay)$, there is no (compact) global attractor.
\item For Example \ref{E10} where we consider the conformally Hamiltonian flow of     $H:\mathbb T\times \mathbb R\to \mathbb R$  defined by $H(\theta, r)=r\sin (2\pi \theta)$ with $\alpha\in (0, 2\pi)$, there is no (compact) global attractor.
\end{itemize}

\end{example}

Can a periodic orbit be attractive, repulsive? \\
It is known for a  non fixed $T$-periodic point a symplectic flow $(\varphi_t)$ that the eigenvalues of $D\varphi_T(x)$ come by pairs, see section 2.2 of \cite{McDuffSalamon2017}. More precisely, $1$ is a eigenvalues with even non-zero multiplicity, and if $\lambda$ is an eigenvalue, then $\lambda$, $\lambda^{-1}$, $\overline{\lambda}$ and $\overline{\lambda}^{-1}$ are eigenvalues with the same mulitiplicity. If $-1$ is an eigenvalue, its multiplicity is even. \\
Hence in the conformal case,  if $\lambda$ is an eigenvalue, then $\overline{\lambda}$, $e^{-\alpha T}\lambda$ and $e^{-\alpha T}\overline{\lambda}$ are eigenvalues with the same multiplicity. Thus, when $\alpha>0$, it may happen that there is an attracting periodic orbit, but there is no repulsive periodic orbit.

Let us now state a maximality property of the global attractor.

\begin{proposition}\label{Pglogattrmax}
When there is a global attractor, all the backward invariant compact subsets are contained in the global attractor.
\end{proposition}

\begin{corollary}\label{Cunstableglobattract}
Assume that $x$ is a periodic hyperbolic periodic point of the flow $(\varphi_t)$ and assume that there is a global attractor. Then the unstable submanifold of $x$ is contained in the global attractor. 
\end{corollary}
When $x$ is not hyperbolic, the statement is true if we replace the unstable submanifold by the   unstable subset, that is, the set of points whose negative orbit tends  toward the orbit of $x$ in negative times.

\begin{corollary} Let $(\varphi_t)$ be a   flow of $M$. Then there is at most one global attractor. 

\end{corollary}

\begin{corollary}
Let $(\varphi_t)$ be a   flow on $M$  that is assumed to be compact. Then the global attractor is $M$.
\end{corollary} 
\begin{proof}[Proof of Proposition \ref{Pglogattrmax}]
We denote the global attractor by $A$ and assume that $K$ is a backward invariant compact subset that is not contained in $A$. We then choose an open subset $U$ such that $A\subset U$, $K\not\subset U$ and $\forall t>0, \varphi_t(\overline U)\subset U$. Thanks to the compactness of $K$  and the forward invariance of $U$, there exists $T>0$ such that $\forall t\geq T, \forall x\in K, \varphi_t(x)\in U$. But since $K$ is backward invariant, we know that $K\subset \varphi_T(K)$, and we obtain that $K\subset \varphi_T(K)\subset U$, which contradicts $K\not\subset U$.

\end{proof}
For the following proposition, we need an hypothesis for the exact symplectic manifold to ensure that   the intersection of a closed exact Lagrangian submanifold and its image by a (symplectic) Hamiltonian diffeomorphism  intersect. The classical hypothesis is to assume that the manifold is convex at infinity, \cite{Floer1988, McDuff1991, Viterbo1999}.

\begin{proposition}\label{PglobalattmeetexLag}
Let $(M, -d\lambda)$ be an exact symplectic manifold that is convex at infinity and let $H:M\to \mathbb R$ be $C^{1, 1}$ and let $\alpha>0$. Assume that $X_H^\alpha$ has a global attractor $A$. Then $A$ meets every closed exact Lagrangian submanifold of $M$.
\end{proposition}
\begin{proof}[Proof of Proposition \ref{PglobalattmeetexLag}] Let $L$ be a closed exact Lagrangian submanifold of $M$ and let $(\varphi_t)$ be the flow of $X_H^\alpha$. 
\begin{lemma}\label{LconformalnotconformaHamisotLag}
Let $L$ be an exact closed Lagrangian submanifold and let $(\varphi_t)$ be a conformally Hamiltonian flow. Let $T\in\mathbb R$ be any real number. Then there exists a (symplectic) Hamiltonian isotopy $(\psi_t)$ such that $  \psi_T(L)=\varphi_T(L)$.
\end{lemma} 
\begin{proof}[Proof of Lemma \ref{LconformalnotconformaHamisotLag}] The proof is very similar to that of  corollary 3 p 172 of  \cite{ArnaudFejoz2024}. We give the main steps of the proof.
\begin{itemize}
\item By examining the proof of Lemma \ref{Lhamexact} in Section \ref{sec:2}, we see that the De Rham cohomology classes of $\lambda_{|L}$ and $\lambda{|\varphi_t(L)}$ satisfy the equality $[\lambda{|\varphi_t(L)}]=e^{-\alpha t}[\lambda_{|L}]$. This implies that   $(\varphi_t(L))$ is an isotopy of exact Lagrangian submanifolds.
\item Let   $t$ be fixed. Let  $j:U\subset T^*L \hookrightarrow M$ be  a Weinstein neighbourhood  of  $\varphi_t(L)$, see \cite{Weinstein1979 }. Note that $j$ is exact symplectic because the zero-section of $T^*L$ and $L\subset M$ are exact Lagrangian submanifolds. Then there exists $\varepsilon>0$ such that for all $s\in [t-\varepsilon, t+\varepsilon]$, $\varphi_s(L)\subset j(U)$ and $j^{-1}(\varphi_s(L))$ is the graph of $du_s$ for some $u_s:L\to\mathbb R$.  The function $u_s$ can be chosen to be $C^1$ with respect to  $s$. 
\item We   then construct a family $(h_s)$ of functions with compact support in  $U$ such that $h_t(q,p)=-\frac{\partial u_t}{\partial t} (q)$ in the neighbourhhod of $L$.  It defines a (symplectic) Hamiltonian isotopy $(f_t)$ with compact support in $U$ such that $f_s(j^{-1}(L))=j^{-1}(\varphi_s(L))$.
\item We now define $H_s:M\to \mathbb R$ by $H_s=h_s\circ j^{-1}$, which defines a (symplectic) Hamiltonian isotopy $(g_s)$ such that $g_s(\varphi_t(L))=\varphi_s(L)$.
\item By transitivity of the relation ``to be  Hamiltonianly isotopic'', we conclude that all  $\varphi_t(L)$ are (symplectically) Hamiltonianly isotopic to $L$.

\end{itemize}

\end{proof}
Now  assume that $L\cap A=\emptyset$. Then there exists an open set $U$ in $M$ such that $A\subset U$, $L\cap U=\emptyset$ and $\forall t>0, \varphi_t(\overline U)\subset U$. Since $L$ is compact, $A$ is the global attractor and $U$ is forward invariant, there exists $T>0$ such that $\forall x\in L, \forall t\geq T, \varphi_t(x)\in U$. We then obtain  that $L\cap\varphi_T(L)=\emptyset$, which contradicts the fact that $L$ and $\varphi_T(L)$ are Hamiltonianly isotopic according to Lemma \ref{LconformalnotconformaHamisotLag}.
\end{proof}

We recall that a function is {\sl coercive} if it tends to $+\infty$ at infinity.

\begin{proposition}[Arnaud-Su-Zavidovique, \cite{ArnaudSuZavidovique2015}] \label{PropExistGlobAttr}
Assume that $H:T^*M\to \mathbb R$ is $C^{1, 1}$,  convex in the fiber direction and coercive. Then  the $(\alpha, H)$ conformal Hamiltonian vector field $X_H^\alpha$ is complete in positive times and has a global attractor.
\end{proposition}

A different proof of this proposition is given in \cite{MaroSorrentino2017} when $H$ is assumed to be Tonelli, which is a more restrictive condition than ours.

\begin{proof}[Proof of Proposition \ref{PropExistGlobAttr}]
 Let $H:T^*M\to\mathbb R$ be $C^{1, 1}$,  convex in the fiber direction and coercive. We prove that there exists $R>0$ such that $H$ is a strict Lyapunov function on $\{ H>R\}$. We can then conclude that $X_H^\alpha$ is complete in positive times and that $\bigcap_{t>0}\varphi_t(\{H\leq R+1\})$ is the global attractor (where $(\varphi_t)$ is the flow of $X_H^\alpha$).
 
 We use canonical coordinates $(q, p)$. 
 \begin{lemma}\label{LemmLyapfunct} Under the assumptions of Proposition \ref{PropExistGlobAttr}, suppose that $H(q, p)>R=\sup_{q\in M} H(q, 0)$ and let  $(q(t), p(t))$  be the orbit of $(q, p)$. Then $$\frac{d}{dt}(H(q(t), p(t))_{|t=0}<0.$$
 \end{lemma}
 \begin{proof}[Proof of Lemma  \ref{LemmLyapfunct}]
 By convexity, we have 
  $$\frac{d}{dt}(H(q(t), p(t))_{|t=0}=-\alpha \partial_pH(q, p)p\leq\alpha(H(q,0)-H(q,p))< 0.$$
 \end{proof}\end{proof}

Up to the end of this section, we will explain different examples. Ma\~n\'e's example is classical in the symplectic setting. Damped mechanical systems are also well-known examples. We will conclude with an example that has a closed, non-isotropic invariant hypersurface that is the global attractor.

\begin{example}[Ma$\tilde{n}$\'e example] Let $N$ be a closed manifold endowed with some Riemannian metric, we denote by $\|\cdot\|$ the corresponding norm and by $\langle\cdot , \cdot\rangle$ the corresponding scalar product.   Let $Y$ be your favorite vector field on $N$. We introduce $Y^\sharp=\langle Y, \cdot\rangle$ that is a section of $T^*N$. Then $H$ is defined on $T^*N$ by
$$H(q,p)=\frac{1}{2} \|p+Y^\sharp(q)\|^2-\frac{1}{2}\| Y^\sharp(q)\|^2$$
and $$X_H^\alpha(q,p)=(\langle p+Y^\sharp(q), \cdot\rangle, {}^tDY^\sharp(q)p-\alpha p).$$
Therefore, the zero-section is invariant and the restricted vector field is $Y$.

Note that if $Y$ is such that, for a certain $q_0$ (in chart), we have ${}^tDY^\sharp(q_0)=\alpha\mathbf 1$ and $Y(q_0)\neq 0$, then $X_H^\alpha(q_0, -Y^\sharp(q_0))=(0,0)$, the flow has a fixed point outside the zero-section. Hence the zero-section is not always the global attractor. In this case, we have a coercive and convex in the fiber  example where the global attractor is not a submanifold.
 
\end{example}

\begin{example}[Mechanical Hamiltonians]  Let $N$ be a closed manifold endowed with some Riemannian metric and let $V:N\to\mathbb R$ be a $C^{1, 1}$ function. The Hamiltonian associated to the potential $-V$ is defined by  
$H(q, p)=\frac{1}{2}\| p\|^2+V(q)$. Let $\alpha>0$ be a positive number.

If $(q(t), p(t))$ is an orbit for $X_H^\alpha$, then we have
$$\frac{d}{dt}H(q(t), p(t))=-\alpha\| p\|^2.$$
So $H$ is a Lyapunov function, which is non increasing along the orbits. Proposition \ref{PropExistGlobAttr} implies that $\omega(q(0), p(0))$ is compact. Then any point $(q_\infty, p_\infty)$ in $\omega(q(0),p(0))$ satisfies $H(q_\infty, p_\infty)=\lim_{t\to +\infty}H(q(t), p(t))$ and therefore $-\alpha\| p_\infty\|^2=0$, that is 
$p_\infty=0$. Therefore,  $\omega(q,p)$ is an invariant set in the  zero-section, and we have $0=\dot p_\infty=-dV(q_\infty)$, that is,  $(q_\infty, p_\infty)$ is a critical point of $H$. We have thus proven that $T^*N$ is the stable subset of the  set of  critical points.

Let $A$ be the global attractor and let $(q_\infty, p_\infty)$ be a point of $A$. Then the negative orbit of $(q_\infty, p_\infty)$ is contained in the compact set $A$, and the $\alpha$-limit set of $(q_\infty, p_\infty)$ is compact. The same proof as before applies in negative times, and therefore $(q_\infty, p_\infty)$ is in the unstable subset  of the set of critical points. We deduce from  Proposition \ref{Pglogattrmax} that $A$ is    the unstable set  of the set of critical points.

In the case of the damped pendulum, this attractor  is not a smooth submanifold but a $C^0$-submanifold (a spiral), see e.g. \cite{ArnHumVit2024}. More generally, when le critical points are all hyperbolic, the global attractor is the union of the unstable submanifolds of the critical points. 

\end{example} 

We now give some ideas for constructing an example where the global attractor is
a submanifold that is not isotropic. The detailed construction is given in Proposition 1 of  \cite{ArnaudFejoz2024}.  Note that this example is very close to an example of Geiges, \cite{Geiges1995}, even though it is different. A recent article of Cieliebak,   Lazarev,   Moreno,  Massoni, \cite{CieLazMasMo2022} suggests that Geiges'example does not contain any exact Lagrangian submanifolds. It would be interesting to know if this is the case of our example. Indeed, when there is no closed exact Lagrangian submanifold,  there is no theory of Birkhoff attractor as developped in \cite{ArnHumVit2024}.
\begin{example}[An example with a global attractor that is a non-isotropic submanifold, Arnaud-Fejoz,  \cite{ArnaudFejoz2024}]\label{ExArnFejnonisot}\phantom{fish}
The idea is to use the suspension of an Anosov flow, as  done in \cite{ArAv1968}.
We consider a  linear  Anosov diffeomorphism $A:\mathbb T^2\to\mathbb T^2$ with two eigenvalues $0<\lambda_-=\frac{1}{\lambda_+}<1<\lambda_+$, for example the one with matrix $\begin{pmatrix} 2&1\\ 1&1\end{pmatrix}$. We choose  $v_\pm$, two eigenvectors associated to $\lambda_\pm$. 

We endow $\mathbb T^2\times \mathbb R$ with the equivalence relation $\sim$ defined by
$$\forall (\xi, z)\in \mathbb T^2\times \mathbb R, (\xi, z)\sim F(\xi, z)=(A\xi, z-1).$$
Then $N$ is the quotient manifold $N=\mathbb T^2\times \mathbb R/\sim$. The vector field $\partial_z$ generates the flow $(\xi, z)\mapsto (\xi, z+t)$. The first return map on $\mathbb T^2\times \{ 0\}$ is $(\xi, 0)\mapsto (A\xi, 0)$. Therefore,  the tangent space to $N$ is the sum of an expanding direction $E_+$ (associated to $\lambda_+$), a contracting direction $E_-$  (associated to $\lambda_-$) and the central direction generated by the vector field $\partial_z$. 

Since we want to construct a conformally symplectic flow, we introduce on $N\times \mathbb R$ a symplectic form $\omega=dq_1\wedge dp_1+dq_2\wedge dp_2$ where the coordinate $q_1$ corresponds to the direction of $E_-$, $p_1$ corresponds to the direction of the vector field $\partial_z$, $q_2$ corresponds to the direction of $E_+$ and $p_2$ to the direction $\mathbb R$. Therfore,  $dq_1\wedge dp_1$ is contracted by $\lambda_-$ by the time-one map, and $dq_1$ is expanded by $\lambda_+=\lambda_-^{-1}$. We therefore define the flow on the $\mathbb R$-coordinate as $p_2\mapsto (\lambda_-)^{2t}p_2$. 

In doing so, we can prove that the manifold is exact symplectic and the vector field is $\log \lambda_+$ times the Liouville vector field, that there exists a global attractor which is $N\times \{ 0\}$, which is 3-dimensional, coisotropic but not isotropic.
\end{example}
\begin{remark}
Before concluding this section, we should mention that in his parallel course \cite{Viterbo2025}, Viterbo introduces the concept of Birkhoff attractor. Despite its name,  it is not generally an attractor. When the assumptions of Proposition \ref{PropExistGlobAttr} are satisfied, the Birkhoff attractor is contained in the global attractor. 
\end{remark} 
\section{Conformal dynamics on locally symplectic manifolds}
\label{sec:4}
We provide two equivalent definitions of the concept of locally symplectic manifold. We
also introduce conformal diffeomorphisms in this context. We provide several examples.
Next, focusing on the conformal Hamiltonian flows, we highlight conservative and dissipative behaviors,
which are closely related to the rotation number (i.e., the  $\eta$-shape) of the
orbit under consideration. We give an example of such a dynamic that has no periodic orbit.
Therefore, in this framework, Weinstein's conjecture is false.

\begin{definition}\phantom{fish}
$\bullet$ A {\sl locally symplectic manifold} is a manifold $M^{(2d)}$ endowed with an atlas $\mathcal A$ of charts $\psi_i:U_i\to\mathbb R^{2d}$ such that on $U_i\cap U_j$, we have $\psi_i^*\omega_\text{stand}=c_{ij}\psi_j^*\omega_\text{stand}$ for certain $c_{ij}>0$ where $\omega_\text{stand}$ is the standard symplectic form on $\mathbb R^{2d}$.\\
$\bullet$ A diffeomorphism $f: M\to M$ is conformal if there exists $c_{ij}(f)>0$ such that 
$$(\psi_i^{-1}\circ  f\circ \psi_j)^*\omega_\text{stand}=c_{ij}(f)\omega_\text{stand}$$
when the expression makes sense.
\end{definition}
We define $B$ as the half-line bundle of the 2-forms $\omega_i=\psi_i^*\omega_\text{stand}$ above $M$ and we consider a section $\Omega$ of this bundle such that, in each chart, we have $\Omega=e^{a_i(x)}\omega_i$. Then $d\Omega=da_i\wedge \Omega$ and on $U_i\cap U_j$, we have $e^{a_i(x)}=c_{ij}e^{a_j(x)}$, from which $a_i(x)=a_j(x)+\log c_{ij}$ and $da_i$ is independent of $i$ and defines a closed 1-form $\eta$ on $M$, such that $d\Omega=\eta\wedge \Omega$.

If we take another section $\Omega'=e^{b(x)}\Omega$, then $d\Omega'=(\eta+db(x))\wedge \Omega'$.

This leads us to introduce the definition of the {\sl conformal symplectic  structure}.

\begin{definition}\phantom{fish}
$\bullet$ A {\sl conformal pair} on the   manifold $M^{(2d)}$ is a pair $(\Omega, \eta)\in \Omega^2(M)\times \Omega^1(M)$ such that $d\eta=0$ and $d\Omega=\eta\wedge\Omega$ and $\Omega^{\wedge d}\neq 0$. \\
$\bullet$ Two such pairs $(\Omega, \eta)$ and $(\Omega_1, \eta_1)$ are equivalent if $(\eta_1, \Omega_1)=(\eta+db, e^b\Omega)$ for some $b:M\to\mathbb R$.\\
$\bullet$ An equivalence class is a {\sl conformal symplectic structure}. \\
$\bullet$ A conformal pair being fixed, $\Omega$ is called the {\sl conformal form} and $\eta$ is the {\sl Lee form}.
\end{definition}
\begin{remark}\phantom{fish}
\begin{enumerate}
\item Given such a a conformal structure, we can be reconstructed an atlas that endows $M$ with a locally symplectic structure. The proof is given in \cite{ChaMur2019}.  
\item If $(\Omega, \eta)$ is a conformal pair and $\theta$ is a primitive of the closed 1-form $\eta$ on an open subset $U$ of $M$, then $e^{-\theta}.\Omega$ is a symplectic form because $d(e^{-\theta}.\Omega)=e^{-\theta}d_\eta \omega=0$. 
In particular, when $\eta$ is exact, $(\omega, \eta=d\theta)$ is  equivalent to the symplectic pair $(e^{-\theta}\Omega, 0)$. 
\item More globally,  consider $\pi_\eta: M_\eta\to M$, which is the lift of $M$ defined by $[\eta]$, i.e. such that for any lift $\pi:\tilde{M}\to M$ such that $\pi^*\eta$ is exact,  there exists $p:\tilde{M}\to M_\eta$ such that $\pi=\pi_\eta\circ p$. Let $\Theta$ be a primitive of $\pi^*_\eta\eta$. Then $\Omega_\eta=e^{-\Theta}\pi_\eta^*\Omega$ is a symplectic form on $M_\eta$.
\item Note that a diffeomorphism $f:M\to M$ is conformal if and only if there exists a function $a:M\to \mathbb R$ such that $f^*\Omega=e^a\Omega$. 
\end{enumerate}
\end{remark}
\begin{exercise} 
Suppose that $(\Omega, \eta)$ is a conformal pair on $M$ such that $\dim M\geq 4$ and that $f^*\Omega=e^a\Omega$.
\begin{enumerate}
\item Prove that   $f^*\eta=\eta+da$. 
\item Let $\pi_\eta: M_\eta\to M$ be the lift of $M$ defined by $[\eta]$ and let $\Theta$ be a primitive of $\pi_\eta^*\eta$. We denote by $\Omega_\eta$ the symplectic form $e^{-\Theta}\pi_\eta^*\Omega$ on $M_\eta$ and by $F:M_\eta\to M_\eta$ a lift of $f$. Using the result of Exercise \ref{ExoLibermann}, prove that there exists a constant $C$ such that $a=\Theta\circ F-\Theta+C$.
\end{enumerate}
\end{exercise} 
Therefore, from now on, a conformal diffeomorphism will be $f:M\to M$ such that there exists $a:M\to \mathbb R$ such that $f^*\Omega=e^a\Omega$ and $f^*\eta=\eta+da$. \\

\noindent{\bf Notation.} If $\eta$ is a closed $1$-form on $M$, the {\sl De Rham-Lichnerowicz differential} $d_\eta:\Omega(M)\to\Omega(M)$ is defined by $d_\eta \nu=d\nu-\eta\wedge\nu$. We then  have $d_\eta\circ d_\eta=0$.

\begin{example}[Examples of locally symplectic manifolds]\phantom{fish}
\begin{enumerate}
\item Any surface endowed with a symplectic form and a closed 1-form is locally symplectic.
\item Let us describe the (twisted) conformal symplectization of a contact manifold as it is introduced in \cite{AllArna2024}. Recall that a contact form on $Y^{(2n+1)}$ is a 1-form $\alpha$ such that everywhere, we have $\alpha\wedge (d\alpha)^{\wedge n}\neq 0$. Let $\beta$ be a closed 1-form on $Y$. The {\sl $\beta$-twisted symplectization } of $(Y, \alpha)$ is $Y\times \mathbb T$ endowed with the Lee form $\eta=\pi^*\beta-d\theta$ and the conformal form $\Omega=-d_\eta(\pi^*\alpha)$ where $\theta:Y\times \mathbb T\to \mathbb T$ and $\pi: Y\times \mathbb T\to Y$ are the projections.\\
For example, $\mathbb S^3\times \mathbb S^1$ (with $\beta=0$) can be endowed with a locally symplectic structure but not with a symplectic form (as $H^2=0$).
\end{enumerate}
\end{example}
\begin{definition} Suppose that the manifold $M$ is endowed with the conformal pair $(\Omega, \eta)$.
Let $H:M\to\mathbb R$ be a $C^{1,1}$  function (meaning that $dH$ is locally Lipchitz). 
\begin{itemize}
\item The {\sl conformally Hamiltonian vector field} is defined by 
$$i_X\Omega =d_\eta H= dH-H.\eta.$$
\item If the flow $(\varphi_t)$ is defined on $\{ x\}\times [0, T]$, the rotation number is 
$$r_T(x)=\int_0^T\eta(X\circ \varphi_t(x))dt.$$
\item When $H=1$, the associated vector field $L_\eta$ that satisfies $i_{L_\eta}\Omega=-\eta$ is called the {\sl Lee vector field}. Since $\eta(L_\eta)=0$,   when it is defined, the rotation number is zero.
\end{itemize} 
\end{definition}
\begin{proposition}\label{Pfactconforme} Let $(\Omega, \eta)$ be a conformal pair on the manifold $M$. Let $(\varphi_t)$ be the conformal Hamiltonian flow associated with a Hamiltonian $H:M\to\mathbb R$. Then, if we denote by $r_t$ the associated rotation number, we have 
$$\varphi_t^*\Omega=e^{r_t}\Omega\text{ and } H\circ \varphi_t=e^{r_t}H.$$
We deduce for the flow $(\psi_t)$ of the Lee vector field that
$$\psi_t^*\Omega=\Omega\text{ and  } H\circ \psi_t=H.$$

\end{proposition} 

\begin{proof}[Proof of Proposition \ref{Pfactconforme}]
We have 
\[\begin{split}
L_X\Omega&= d(i_X\Omega)+i_X(d\Omega)=d(dH-H\eta)+i_X(\eta\wedge\Omega)\\
&=-dH\wedge \eta +\eta(X)\Omega-\eta\wedge i_X\Omega\\
&=\eta\wedge dH+\eta(X)\Omega-\eta\wedge(dH-H\eta)=\eta\wedge dH+\eta(X)\Omega-\eta\wedge dH\\
&= \eta(X)\Omega
\end{split}\]
which gives the first formula.\\
From $i_X\Omega=dH-H\eta$ we deduce $dH(X)=\eta(X)H$ and this gives the second formula.

\end{proof} 
\begin{exercise}
Prove that we  also have: $\eta\circ \varphi_t=\eta+dr_t$.
\end{exercise}

\begin{proposition}\label{PconfpairHam}
Let $(\Omega, \eta)$ be a conformal pair on the manifold $M$.
\begin{enumerate}
\item If $\eta$ is not exact, the map $\mathcal H_{\Omega, \eta}:H\mapsto X_H$, that maps a  function onto the associated Hamiltonian vector field, is injective. 
\item If $(\Omega', \eta')$ is a different  conformal pair but  equivalent to $(\Omega, \eta)$, then the map $\mathcal H_{\Omega', \eta'}$ is not $\mathcal H_{\Omega, \eta}$, but both maps have the same image. More precisely, we have
$$\mathcal X_{\Omega, \eta}(H)=\mathcal X_{e^f\Omega, \eta+df}(e^fH).$$
\item When $H>0$, $X_H$ is the Lee vector field for an equivalent conformal pair.
\end{enumerate}
\end{proposition}

\begin{example}\label{ExampleHamconf}\phantom{fish} \begin{enumerate}
\item We consider on $\mathbb T^2$ the conformal pair $(\Omega, \eta)=(d\theta_1\wedge d\theta_2, -2\pi d\theta_1)$. We consider the  following two Hamiltonians.

\begin{multicols}{2}
$H(\theta_1, \theta_2)=\sin 2\pi\theta_1$\\
$X=(0, -2\sqrt{2}\pi\sin 2\pi(\frac{1}{8}+\theta_1))$\\

\begin{tikzpicture}
        \draw[->] (-1,1) -- (-1,0) ;
         \draw  (-1,0) -- (-1,-1) node[below]{$0$} ;
         \draw[->] (-0.75,1)  -- (-0.75,0.05) ;
          \draw[->] (-0.75,0.05) -- (-0.75,-0.05) ;
         \draw  (-0.75,-0.05) -- (-0.75,-1)  ;
          \draw[->] (-0.5,1) -- (-0.5,0) ;
         \draw  (-0.5,0) -- (-0.5,-1)  ;
          \draw  [dashed]  (-0.25,1) -- (-0.25,-1)   ;
          \draw[->] (0,-1)node[below]{$1/2$} -- (0,0) ;
         \draw  (0,0) -- (0,1) ;
         \draw[->] ( 0.25,-1) -- (0.25,-0.05) ;
          \draw[->] (0.25,-0.05) -- (0.25, 0.05) ;
         \draw  (0.25,0.05) -- (0.25,1)   ;
          \draw[->] ( 0.5,-1) -- (0.5,0) ;
          \draw (0.5, 0)  -- (0.5, 1) ;
          \draw [dashed]  (0.75,-1)--(0.75,1)   ;
          \draw[->] (1,1)--(1,0);
          \draw (1,0)--(1,-1) node[below]{$1$} ;
           \end{tikzpicture}
           
           $H(\theta_1, \theta_2)=\sin 2\pi\theta_2$\\
$X=(2\pi\cos 2\pi\theta_2, -2\pi\sin 2\pi \theta_2)$\\

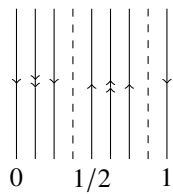
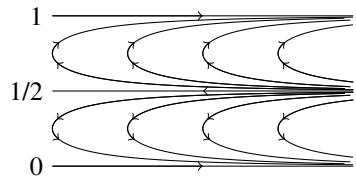
\begin{tikzpicture} 
\draw[->] ( -2,1) -- (0,1) ;
\draw( -2,1) -- (2,1) ;
\draw[->][domain=-2:-0.25][samples=200] plot [variable=\x]  ({\x*\x-2},{0.5+((exp(\x)-exp(-\x))/((2*(exp(\x))+(2*exp(-\x))))});
\draw[->][domain=-0.25:0.3][samples=200] plot [variable=\x]  ({\x*\x-2},{0.5+((exp(\x)-exp(-\x))/((2*(exp(\x))+(2*exp(-\x))))});
\draw[domain=-2:2][samples=200] plot [variable=\x]  ({\x*\x-2},{0.5+((exp(\x)-exp(-\x))/((2*(exp(\x))+(2*exp(-\x))))});
\draw[->][domain=-1.7:-0.25][samples=200] plot [variable=\x]  ({\x*\x-1},{0.5+((exp(\x)-exp(-\x))/((2*(exp(\x))+(2*exp(-\x))))});
\draw[->][domain=-0.25:0.3][samples=200] plot [variable=\x]  ({\x*\x-1},{0.5+((exp(\x)-exp(-\x))/((2*(exp(\x))+(2*exp(-\x))))});
\draw[domain=-1.7:1.7][samples=200] plot [variable=\x]  ({\x*\x-1},{0.5+((exp(\x)-exp(-\x))/((2*(exp(\x))+(2*exp(-\x))))});
\draw[->][domain=-1.4:-0.25][samples=200] plot [variable=\x]  ({\x*\x},{0.5+((exp(\x)-exp(-\x))/((2*(exp(\x))+(2*exp(-\x))))});
\draw[->][domain=-0.25:0.3][samples=200] plot [variable=\x]  ({\x*\x},{0.5+((exp(\x)-exp(-\x))/((2*(exp(\x))+(2*exp(-\x))))});
\draw[domain=-1.4:1.4][samples=200] plot [variable=\x]  ({\x*\x},{0.5+((exp(\x)-exp(-\x))/((2*(exp(\x))+(2*exp(-\x))))});
\draw[->][domain=-1:-0.25][samples=200] plot [variable=\x]  ({\x*\x+1},{0.5+((exp(\x)-exp(-\x))/((2*(exp(\x))+(2*exp(-\x))))});
\draw[->][domain=-0.25:0.3][samples=200] plot [variable=\x]  ({\x*\x+1},{0.5+((exp(\x)-exp(-\x))/((2*(exp(\x))+(2*exp(-\x))))});
\draw[domain=-1:1][samples=200] plot [variable=\x]  ({\x*\x+1},{0.5+((exp(\x)-exp(-\x))/((2*(exp(\x))+(2*exp(-\x))))});

\draw[->] ( 2,00) -- (0,0) ;
\draw (0,0)--(-2, 0);
\draw[->][domain=-2:-0.25][samples=200] plot [variable=\x]  ({\x*\x-2},{-0.5+((exp(-\x)-exp(\x))/((2*(exp(\x))+(2*exp(-\x))))});
\draw[->][domain=-0.25:0.3][samples=200] plot [variable=\x]  ({\x*\x-2},{-0.5+((exp(-\x)-exp(\x))/((2*(exp(\x))+(2*exp(-\x))))});
\draw[domain=-2:2][samples=200] plot [variable=\x]  ({\x*\x-2},{-0.5+((exp(-\x)-exp(\x))/((2*(exp(\x))+(2*exp(-\x))))});
\draw[->][domain=-1.7:-0.25][samples=200] plot [variable=\x]  ({\x*\x-1},{-0.5+((exp(-\x)-exp(\x))/((2*(exp(\x))+(2*exp(-\x))))});
\draw[->][domain=-0.25:0.3][samples=200] plot [variable=\x]  ({\x*\x-1},{-0.5+((exp(-\x)-exp(\x))/((2*(exp(\x))+(2*exp(-\x))))});
\draw[domain=-1.7:1.7][samples=200] plot [variable=\x]  ({\x*\x-1},{-0.5+((exp(-\x)-exp(\x))/((2*(exp(\x))+(2*exp(-\x))))});
\draw[->][domain=-1.4:-0.25][samples=200] plot [variable=\x]  ({\x*\x},{-0.5+((exp(-\x)-exp(\x))/((2*(exp(\x))+(2*exp(-\x))))});
\draw[->][domain=-0.25:0.3][samples=200] plot [variable=\x]  ({\x*\x},{-0.5+((exp(-\x)-exp(\x))/((2*(exp(\x))+(2*exp(-\x))))});
\draw[domain=-1.4:1.4][samples=200] plot [variable=\x]  ({\x*\x},{-0.5+((exp(-\x)-exp(\x))/((2*(exp(\x))+(2*exp(-\x))))});
\draw[->][domain=-1:-0.25][samples=200] plot [variable=\x]  ({\x*\x+1},{-0.5+((exp(-\x)-exp(\x))/((2*(exp(\x))+(2*exp(-\x))))});
\draw[->][domain=-0.25:0.3][samples=200] plot [variable=\x]  ({\x*\x+1},{-0.5+((exp(-\x)-exp(\x))/((2*(exp(\x))+(2*exp(-\x))))});
\draw[domain=-1:1][samples=200] plot [variable=\x]  ({\x*\x+1},{-0.5+((exp(-\x)-exp(\x))/((2*(exp(\x))+(2*exp(-\x))))});\draw[->] ( -2,-1) -- (0,-1) ;
\draw( -2,-1) -- (2,-1) ;

\draw (-2,1) node[left]{1} ;
\draw (-2,0) node[left]{1/2} ;
\draw (-2,-1) node[left]{0} ;
\end{tikzpicture}
  
\end{multicols}
Note that the picture on the left ressembles a (symplectic) Hamiltonian dynamics, but the picture on the right  is highly dissipative, with an  attractive periodic orbit and one repulsive periodic orbit.
\item Let $(Y,\alpha)$ be a contact manifold and let $\beta$ be a closed 1-form on $Y$. Let $H:Y\to \mathbb R$ be a function, we recall that the associated contact vector field $X$ is defined by
$$\alpha(X)=H\text{ and }  i_Xd\alpha=(dH.R)\alpha-dH$$
where $R$ is the Reeb vector-field (associated with the constant Hamiltonian 1). On the $\beta$ twisted conformal symplectization of $(Y, \alpha)$, the associated to $\tilde H(x,\theta)=H(x)$ conformal vector field   is
$$\tilde X=X\oplus (\beta(X)-dH.R)\partial_\theta\in TY\oplus T\mathbb T$$
and the Lee vector field is $R\oplus \beta(R)\partial_\theta$.

\end{enumerate}

\end{example}
\begin{proof}[Proof of Proposition \ref{PconfpairHam}]
\begin{enumerate}
\item We assume that  $\eta$ is not exact and consider $\pi_\eta: M_\eta\to M$ which is the lift of $M$ defined by $[\eta]$. We want to prove that the kernel of the map $\mathcal H_{\Omega, \eta}:H\mapsto X_H$ is $\{ 0\}$. Let $H\in \ker \mathcal H_{\Omega, \eta}$.  Let $\theta$ be a primitive of $\pi_\eta^*\eta$. Then we have 
$d(e^{-\theta}H\circ \pi)=0$ so $e^{-\theta}H\circ \pi$ is constant. Since $\eta$   is not exact, there exists a sequence $(y_n)$ in $M_\eta$ such that $\lim_{n\to\infty}\theta(y_n)=+\infty$. We deduce that $H=0$ and therefore that   $\mathcal H_{\Omega, \eta}$ is injective.
\item We leave it to the reader to verify the formula
$$\mathcal X_{\Omega, \eta}(H)=\mathcal X_{e^f\Omega, \eta+df}(e^fH).$$
\item When $H>0$, we define $f=-\log H$. The previous formula gives
$$\mathcal X_{\Omega, \eta}(H)=\mathcal X_{\frac{1}{H}\Omega, \eta-\frac{dH}{H}}(1).$$
Thus,  $X_H$ is the Lee vector field for the  equivalent conformal pair $(\frac{1}{H}\Omega, \eta-\frac{dH}{H})$.
\end{enumerate}
\end{proof}
We endow $\mathbb T^2$ with its usual flat Riemannian metric. 
\begin{proposition}\label{PNoWeinstein} Let $\beta =a_1dx_1+a_2dx_2$ be a constant 1-form defined on $\mathbb T^2$ such that $1$, $a_1$ and $a_2$ are rationally independent. Let $T^1\mathbb T^2$ be the unit tangent vector bundle of $\mathbb T^2$, which is a contact manifold and let us consider its $\beta$ twisted conformal symplectization. Then the Lee vector field has no periodic orbit.
\end{proposition}  
Therefore,  there is no analogue   to Weinstein's conjecture in the conformally symplectic setting.

\begin{proof}[Proof of Proposition \ref{PNoWeinstein} ]
There is a natural diffeomorphism between $T^1\mathbb T^2$ and $\mathbb T^2\times \mathbb S^1$. Then the Reeb flow is the geodesic flow  given by
$\varphi_t^R(x, v)=(x+tv, v)$. Furthermore,  we have
$\int_0^t\beta(\partial_s\varphi_s^R(x,v))ds=t(a_1v_1+a_2v_2)$ and thus the Lee flow is
$$\psi_t(x, v, \theta)=(x+tv, v, \theta+t(a_1v_1+a_2v_2)).$$
The orbit is $T$-periodic if $Tv\in\mathbb Z^2$ and $T(a_1v_1+a_2v_2)\in\mathbb Z$, that is,  if $Tv_1\in\mathbb Z$, $Tv_2\in\mathbb Z$ and $(Tv_1)a_1+(Tv_2)a_2\in \mathbb Z$, which is impossible for $T\neq 0$
because $1$, $a_1$ and $a_2$ are rationally independent. 

\end{proof}

Looking to the first point in Example \ref{ExampleHamconf}, we see that
\begin{itemize}
\item in the  picture on the right, all orbits are periodic and the integral of $\eta$ along such an orbit is zero; no orbit is repulsive or attractive;
\item in the  picture on the   left, there are two periodic orbits, one repulsive and the other attractive; moreover, the integral of $\eta$ along such a periodic orbit is non-zero.
\end{itemize}
Let us explain why.
\begin{proposition}\label{Pexporbper} Assume that $\gamma:\mathbb R/T\mathbb Z\to M$ is a periodic orbit of the conformally Hamiltonian vector field $X$. Denote
$$\overline r(\gamma)=\frac{1}{T}\int_0^T\eta\circ X(\gamma(s))ds.$$
The characteristic multipliers $\lambda_1, \dots, \lambda_{2n}$ satisfy the relations
$$\lambda_{n-p}\overline\lambda_{n+p+1}=e^{T\overline r(\gamma)}.$$
Moreover, if $\overline r(\gamma)\neq 0$, then $H(\gamma(0))=0$.
\end{proposition}

\begin{proof}[Proof of Proposition \ref{Pexporbper} ] We deduce from Proposition \ref{Pfactconforme}  that 
$$\varphi_T^*\Omega_{\vert\gamma(0)}=e^{T\overline{r}(\gamma)}\Omega_{\vert\gamma(0)}\text{ and } H\circ \varphi_T(\gamma(0))=e^{T\overline{r}(\gamma)}H(\gamma(0)).$$
Since $\varphi_T(\gamma(0))=\gamma(0)$, we deduce that $H(\gamma(0))=0$ when $\overline r(\gamma)\neq 0$. .

The linear isomorphism $e^{-\frac{T}{2}\overline{r}(\gamma)}D\varphi_T(\gamma(0)):(T_{\gamma(0)}M, \Omega(\gamma(0)))\to (T_{\gamma(0)}M, \Omega(\gamma(0))$ is symplectic, so if $\lambda$ is an eigenvalue of $D\varphi_T(\gamma(0))$, then $\overline{\lambda}$, $e^{T\overline{r}(\gamma)}\lambda^{-1}$ and $e^{T\overline{r}(\gamma)}\overline{\lambda}^{-1}$ are also eigenvalues of $D\varphi_T(\gamma(0))$ with the same multiplicity.

\end{proof}

\begin{remark}\phantom{fish}
\begin{enumerate}
\item There is therefore a  relation between the Lyapunov spectrum of a periodic orbit and its $\eta$-shape.
\item There exists a similar statement, ~\cite{LiveraniWojtkowski1998,AllArna2024}, for the Lyapunov exponents of the ergodic invariant measures.
\item There are some examples of attractive periodic orbits,  \cite{AllArna2024}.
\item Let $x$ be a fixed point of a conformally Hamiltonian flow, i.e. such that $X(x)=0$. Let $\theta$ be a primitive of $\eta$ in some neighborhood $V$ of $x$. Then $x$ is a critical point of $e^{-\theta}H$ and $e^{-\theta}H$ is constant along the orbits contained in $V$. Suppose then that $x$ is a local strict minimum of $e^{-\theta}H$. Then the (local) level sets of $e^{-\theta}H$ are   invariant under the flow and the local dynamics is conjugate to a symplectic dynamics close to some elliptic equilibrium. 
\end{enumerate}

\end{remark}
Consequently, periodic orbits can be either dissipative or conservative. To conclude this course, we mention a result that explains the decomposition of M into conservative and dissipative parts. A proof is given in  \cite{AllArna2024}.

\begin{theorem} Let $M$ be a compact locally symplectic manifold and let $(\varphi_t)$ be a conformally Hamiltonian flow. Then there exists a measurable decomposition of $M=E\cup G$ such that
\begin{itemize}
\item almost all point in $E$ are positively and negatively recurrent;
\item all points of $G$ satisfy 
$$\lim_{t\to \pm\infty}r_t(x)=-\infty\text{ and } \alpha(x)\cup\omega(x)\subset \{ H=0\}.$$
\end{itemize}

\end{theorem}

\begin{acknowledgement}
I would like to warmly thank the organizers for this wonderful school and the audience for their patience and insightful questions.
\end{acknowledgement}
%

%
%
%

\end{document}